\documentclass[12pt]{article}

\usepackage{bookmark}
\usepackage{amsthm}
\usepackage{amssymb}
\usepackage{amsmath}
\usepackage{amsfonts}
\usepackage{cancel}
\usepackage{times}
\usepackage{bbm}
\usepackage{cite}
\usepackage{color, soul}
\usepackage{wasysym}
\usepackage{hyperref}

\hypersetup{
     colorlinks   = true,
     citecolor    = black,
     linkcolor    = blue
}

\usepackage[paper=a4paper, top=2.5cm, bottom=2.5cm, left=2.5cm, right=2.5cm]{geometry}

\def\qed{\hfill $\vcenter{\hrule height .3mm
\hbox {\vrule width .3mm height 2.1mm \kern 2mm \vrule width .3mm
height 2.1mm} \hrule height .3mm}$ \bigskip}

\def \RR {\mathbb R}
\def \NN {\mathbb N}
\def \EE {\mathbb E}

\def \vphi {\varphi}

\def \stensor {\scriptsize \astrosun}
\newcommand\norm[1]{\left\lVert#1\right\rVert}
\newtheorem{theorem}{Theorem}
\newtheorem{lemma}{Lemma}

\newtheorem{corollary}[theorem]{Corollary}
\theoremstyle{definition}

\theoremstyle{remark}

\long\def\symbolfootnotetext[#1]#2{\begingroup
\def\thefootnote{\fnsymbol{footnote}}\footnotetext[#1]{#2}\endgroup}

\title{A CLT in Stein's distance for generalized Wishart matrices and higher order tensors}

\author{Dan Mikulincer\thanks{Supported by an Azrieli foundation fellowship. Email: danmiku@gmail.com.}\\
	\text{Weizmann Institute of Science}}

\begin{document}
	
	\maketitle
	
	\begin{abstract}
	We study the central limit theorem for sums of independent tensor powers, $\frac{1}{\sqrt{d}}\sum\limits_{i=1}^d X_i^{\otimes p}$. We focus on the high-dimensional regime where $X_i \in \RR^n$ and $n$ may scale with $d$. Our main result is a proposed threshold for convergence. Specifically, we show that, under some regularity assumption, if $n^{2p-1}\ll d$, then the normalized sum converges to a Gaussian. The results apply, among others, to symmetric uniform log-concave measures and to product measures. This generalizes several results found in the literature.
	Our main technique is a novel application of optimal transport to Stein's method which accounts for the low dimensional structure which is inherent in $X_i^{\otimes p}$.
	\end{abstract}

\section{Introduction}
Let $\mu$ be an isotropic\footnote{That is, $\mu$ is centered and its covariance matrix is the identity.} probability measure on $\RR^n$. For $2\leq p \in \NN$, we consider the following tensor analogue of the Wishart matrix,
$$\frac{1}{\sqrt{d}}\sum\limits_{i=1}^d\left(X_i^{\stensor p} - \EE\left[X_i^{\stensor p} \right]\right),$$
where $X_i\sim \mu$ are i.i.d. and $X_i^{\stensor p}$ stands for the symmetric $p$'th tensor power of $X_i$. We denote the law of this random tensor by $W_{n,d}^p(\mu)$. Such distributions arise naturally as the sample moment tensor of the measure $\mu$, in which case $d$ serves as the sample size. For reasons soon to become apparent, we will sometimes refer to such tensors as \emph{Wishart tensors}.\\

When $p=2$, $W^2_{n,d}(\mu)$ is the sample covariance of $\mu$. If
$\mathbb{X}$ is an $n \times d$ matrix with columns independently distributed as $\mu$,  then $W^2_{n,d}(\mu)$ may also be realized as the upper triangular part of the matrix,
\begin{equation} \label{eq: wishart representation}
\frac{\mathbb{X}\mathbb{X}^T-d\mathrm{Id}}{\sqrt{d}}.
\end{equation} Hence, $W_{n,d}^2(\mu)$ has the law of a Wishart matrix. These matrices have recently been studied in the context of random geometric graphs (\cite{bubeck16testing,bubeck19entropic,eldan2016information,brennan2019phase}).\\

For fixed $p,n$, according to the central limit theorem (CLT), as $d\to\infty$, $W^p_{n,d}(\mu)$ approaches a normal law. The aim of this paper is to study the \emph{high-dimensional regime} of the problem, where we allow the dimension $n$ to scale with the sample size $d$. Specifically, we investigate possible conditions on $n$ and $d$ for the CLT to hold. Observe that this problem may be reformulated as a question about the rate of convergence in the high-dimensional CLT, for the special case of Wishart tensors. \\

Our starting point is the paper \cite{bubeck19entropic}, which obtained an optimal bound when $p = 2$, for log-concave product measures. Remark that when $\mu$ is a product measure, the entries of the matrix $\mathbb{X}$ in \eqref{eq: wishart representation} are all independent. The proof \cite{bubeck19entropic} was information-theoretic and made use of the chain rule for relative entropy to account for the low-dimensional structure of $W^2_{n,d}(\mu)$. For now, we denote $\widetilde{W}^2_{n,d}(\mu)$ to be the same law as $W^2_{n,d}(\mu)$, but with the diagonal elements removed (see below for a precise definition).
 \begin{theorem}[{\cite[Theorem 1]{bubeck19entropic}}] \label{thm:sebtheorem}
 	Let $\mu$ be a log-concave product measure on $\RR^n$ and let $\gamma$ denote the standard Gaussian in $\RR^{\binom{n}{2}}$. Then, 
 	\begin{enumerate}
 		\item If $n^3 \ll d$ then $\mathrm{Ent}\left(\widetilde{W}^2_{n,d}(\mu)||\gamma\right) {\xrightarrow{n \to \infty}} 0.$
 		\item If $n^3 \gg d$, then $W^2_{n,d}(\mu)$ remains bounded away from any Gaussian law.
 	\end{enumerate}
 Here, $\mathrm{Ent}$ stands for relative entropy (see Section \ref{sec:defs} for the definition).
 \end{theorem}
 Thus, for log-concave product measures there is a sharp condition for the CLT to hold. Our results, which we now summarize, generalize Point 1 of Theorem \ref{thm:sebtheorem} in several directions and are aimed to answer questions which were raised in \cite{bubeck19entropic}.
 \begin{itemize}
 	\item We show that it is not necessary for $\mu$ to have a product structure. So, in particular, the matrix $\mathbb{X}$ in \eqref{eq: wishart representation} may admit some dependence between its entries.
 	\item If $\mu$ is a product measure, we relax the log-concavity assumption and show the same result holds for a much larger class of product measures.
 	\item The above results extend to the case $p > 2$, and we propose the new threshold $n^{2p-1} \ll d$.
 	\item We show that Theorem \ref{thm:sebtheorem} is still true when we take the full symmetric tensor $W^2_{n,d}(\mu)$ and include the diagonal.
 \end{itemize}
Our method is based on a novel application of Stein's method.  Stein's theory is a prominent set of techniques which was developed in order to answer questions related to convergence rates along the CLT. The method was first introduced in \cite{stein72bound,stein86approximate} as a way to estimate distances to the normal law. Since then, it had found numerous applications in studying the quantitative central limit theorem, also in high-dimensions (see \cite{ross11fundmentals} for an overview). \\

Na\"ively, since $W_{n,d}^p(\mu)$ can be realized as a sum of i.i.d. random vectors, one should be able to employ standard techniques of Stein's method (such as exchangeable pairs \cite{chatterjee08multivariate}, as proposed in \cite{bubeck19entropic}) in order to deduce some bounds. However, it turns out that the obtained bounds are sub-optimal. The reason for this sub-optimality is that, while $X^{\stensor p}$ is a random vector in a high-dimensional space, its randomness comes from the lower-dimensional $\RR^n$. So, at least on the intuitive level, one must exploit the low-dimensional structure of the random tensor in order to produce better bounds. Our method is particularly adapted to this situation and may be of use in other, similar, settings.\\
\subsection{Definitions and notations}  \label{sec:defs}
\paragraph{Distance to the standard Gaussian:} Let $\gamma_n$ stand for the normal law on $\RR^n$ with density
$$\frac{d\gamma_n}{dx}(x) = \frac{1}{(2\pi)^{\frac{n}{2}}}e^{-\frac{\norm{x}^2_2}{2}}.$$
We will sometime omit the subscript, when the dimension is obvious from the context. Fix now a measure $\mu$ in $\RR^n$. For any $m > 0$, the $m$-Wasserstein's distance between $\mu$ and $\gamma$ is defined by
$$\mathcal{W}_m(\mu,\gamma) = \sqrt[m]{\inf\limits_{\pi} \int \norm{x-y}_2^md\pi(x,y)},$$
where the infimum is taken over all couplings $\pi$, of $\mu$ and $\gamma$. Another notion of distance is that of relative entropy, which is given by
$$\mathrm{Ent}(\mu||\gamma) = \int\log\left(\frac{d\mu}{d\gamma}\right)d\mu.$$
It is known that relative entropy bounds the quadratic Wasserstein distance to the the Gaussian via Talagrand's transportation-entropy inequality (\cite{talagrand96transportation}) as well as controlling the total variation distance through Pinsker's inequality (\cite{cover06elements}).

\paragraph{Probability measures:} A measure $\mu$  is said to be isotropic if it is centered and its covariance is the identity. I.e.
$$\int xd\mu(x) = 0, \ \ \ \ \  \int x\otimes xd\mu(x) = \mathrm{Id}.$$
We say that a measure $\mu$ is log-concave, if it has a density, twice differentiable almost everywhere, for which
$$-\mathrm{Hessian}\left(\log\left(\frac{d\mu}{dx}\right)\right) \succeq 0,$$
where the inequality is between positive semi-definite matrices.
If instead, for some $\sigma>0$
$$-\mathrm{Hessian}\left(\log\left(\frac{d\mu}{dx}\right)\right) \succeq \sigma\cdot\mathrm{Id},$$
we say that $\mu$ is $\sigma$-uniformly log-concave.
A measure is said to be unconditional, if its density satisfies
$$\frac{d\mu}{dx}(\pm x_1,...,\pm x_n) = \frac{d\mu}{dx}(|x_1|,...,|x_n|),$$
where in the left side of the equality we consider all possible sign patterns. Note that, in particular, if $X = (X_{(1)},...,X_{(n)})$ is isotropic and unconditional, then , for any choice of distinct indices $j_1,...,j_k$ and powers $n_2,...,n_k$,
\begin{equation} \label{eq: unconditional moments}
\EE\left[X_{(j_1)}\cdot X_{(j_2)}^{n_2}\cdot X_{(j_3)}^{n_3}\cdot...\cdot X_{(j_k)}^{n_k}\right] = 0.
\end{equation}
Finally, if $\vphi: \RR^n \to \RR^N$ for some $N \geq 0$, we denote $\vphi_*\mu$, to be the push-forward of $\mu$ by $\vphi$.\\
\paragraph{Tensor spaces:}  Fix $\{e_j\}_{j=1}^n$ to be the standard basis in $\RR^n$. We identify the tensor space $\left(\RR^n\right)^{\otimes p}$ with $\RR^{n^p}$ where the base is given by
$$\{e_{j_1}e_{j_2}...e_{j_p}|1\leq j_1, j_2,..., j_p\leq n\}.$$
Under this identification, we may consider the symmetric tensor space $\mathrm{Sym}^p(\RR^n) \subset \left(\RR^n\right)^{\otimes p}$ with basis
$$\{e_{j_1}e_{j_2}...e_{j_p}|1\leq j_1\leq j_2...\leq j_p\leq n\}.$$
We will also be interested in the subspace of principal tensors, $\widetilde{\mathrm{Sym}}^p(\RR^n) \subset \mathrm{Sym}^p(\RR^n)$, spanned by the basis elements
$$\{e_{j_1}e_{j_2}...e_{j_p}|1\leq j_1< j_2...< j_p\leq n\}.$$
Our main result will deal with the marginal of $W_{n,d}^p(\mu)$ on the subspace $\widetilde{\mathrm{Sym}}^p(\RR^n)$. We denote this marginal law by
$\widetilde{W}_{n,d}^p(\mu)$. Put differently, if $X_i = (X_{i,1},...,X_{i,n})$ are i.i.d. random vectors with law $\mu$. Then, $\widetilde{W}_{n,d}^p(\mu)$ is the law of a random vector in $\widetilde{\mathrm{Sym}}^p(\RR^n)$ with entries
$$\left(\frac{1}{\sqrt{d}}\sum\limits_{i=1}^d X_{i,j_1}\cdot X_{i,j_2}\cdot\dots\cdot X_{i,j_p}\right)_{1\leq j_1 <\dots<j_p\leq n}.$$

Throughout this paper we use $C,C',c,c'...$ to denote absolute positive constants whose value might change between expressions. In case we want to signify that the constant might depend on some parameter $a$, we will write $C_a,C'_a$.
\subsection{Main results}
Our main contribution is a new approach, detailed in Section \ref{sec:method}, to Stein's method, which allows to capitalize on the fact that a high-dimensional random vector may have some latent low-dimensional structure. Thus, it is particularly well suited to study the CLT for $W_{n,d}^p(\mu)$. Using this approach, we obtain the following threshold for the CLT: Suppose that $\mu$ is a "nice" measure. Then, if $n^{2p-1} \ll d$, $W_{n,d}^p(\mu)$ is approximately Gaussian, as $d$ tends to infinity. \\

We now state several results which are obtained using our method. The first result shows that, under some assumptions, the matrix $\mathbb{X}$ in \eqref{eq: wishart representation}, can admit some dependencies, even when considering higher order tensors. 
\begin{theorem} \label{thm: uniform measures}
	Let $\mu$ be an isotropic $L$-uniformly log-concave measure on $\RR^n$ which is also unconditional. Denote $\Sigma^{-\frac{1}{2}} = \sqrt{\widetilde{\Sigma}_p(\mu)^{-1}}$, where $\widetilde{\Sigma}_p(\mu)$ is the covariance matrix of $\widetilde{W}^p_{n,d}(\mu)$. Then, there exists a constant $C_p$, depending only on $p$, such that
	$$\mathcal{W}_2^2\left(\Sigma^{-\frac{1}{2}}_*\widetilde{W}_{n,d}^p(\mu),\gamma\right) \leq  \frac{C_p}{L^{4}}\frac{n^{2p-1}}{d},$$
	where $\Sigma^{-\frac{1}{2}}_*\widetilde{W}_{n,d}^p(\mu)$ is the push-forward by the linear operator $\Sigma^{-\frac{1}{2}}$.
\end{theorem}
An important remark, which applies to the coming results as well, is that the bounds are formulated with respect to the quadratic Wasserstein distance. However, as will become evident from the proof, the bounds actually hold with a stronger notion of distance: namely, Stein's discrepancy (see Section \ref{sec:perlim} for the definition). We have decided to state our results with the more familiar Wasserstein distance to ease the presentation.
Our next result is a direct extension of Theorem \ref{thm:sebtheorem}, as it both applies to a larger class of product measures and to $p > 2$.
\begin{theorem} \label{thm: spec measures}
	Let $\mu$ be an isotropic product measure on $\RR^n$, with independent coordinates. Then, there exists a constant $C_p > 0$, depending only on $p$, such that
	\begin{enumerate}
		\item If $\mu$ is log-concave, then
		$$\mathcal{W}_2^2\left(\widetilde{W}_{n,d}^p(\mu), \gamma\right) \leq C_p\frac{n^{2p-1}}{d}\log(n)^{2}.$$
		\item If each coordinate marginal of $\mu$ satisfies the $L_1$-Poincar\'e inequality (see Section \ref{subsec:smooth}) with constant $c>0$, then
		$$\mathcal{W}_2^2\left(\widetilde{W}_{n,d}^p(\mu), \gamma\right) \leq C_p \frac{1}{c^{2p + 2}}\frac{n^{2p-1}}{d}\log(n)^{4}.$$
		\item If there exists a uni-variate polynomial $Q$ of degree $k$, such that each coordinate marginal of $\mu$ has the same law as the push-forward measure $Q_*\gamma_1$, then
		$$\mathcal{W}_2^2\left(\widetilde{W}_{n,d}^p(\mu),\gamma\right) \leq C_{Q,p}\frac{n^{2p-1}}{d}\log(n)^{2(k-1)},$$
		where $C_{Q,p} > 0$ may depend both on $p$ and the polynomial $Q$.
	\end{enumerate}
\end{theorem}
Observe that, when $\mu$ is an isotropic product measure, then $\widetilde{W}_{n,d}^p(\mu)$ is also isotropic (when considered as a random vector in $\widetilde{\mathrm{Sym}}^p(\RR^n)$), which explains why the matrix $\Sigma^{-\frac{1}{2}}$ does not appear in Theorem \ref{thm: spec measures}.
 Our last result is an extension to Theorem \ref{thm: spec measures} which shows that, sometimes, we may consider subspaces of $\left(\RR^n\right)^{\otimes p}$ which are strictly larger than $\widetilde{\mathrm{Sym}}^p\left(\RR^n\right)$. We specialize to the case $p=2$, and show that one may consider the full symmetric matrix $W_{n,d}^2(\mu)$.
\begin{theorem} \label{thm: p2}
	Let $\mu$ be an isotropic log-concave measure on $\RR^n$. Assume that $\mu$ is a product measure with independent coordinates and denote $\Sigma^{-\frac{1}{2}} = \sqrt{\Sigma_2(\mu)^{-1}}$, where $\Sigma_2(\mu)$ is the covariance matrix of $W^2_{n,d}(\mu)$. Then, there exists a universal constant $C > 0$ such that
	$$\mathcal{W}^2_2\left(\Sigma^{-\frac{1}{2}}_*W_{n,d}^2(\mu),\gamma\right) \leq C\frac{n^{3}}{d}\log(n)^{2}.$$
\end{theorem}
\subsection{Related work}
The study of normal approximations for high-dimensional Wishart tensors was initiated in \cite{bubeck16testing} (see \cite{tiefeng15approximation} as well, for an independent result), which dealt with the case of $\widetilde{W}_{n,d}^2(\gamma)$. The authors were interested in detecting latent geometry in random geometric graphs. The main result of \cite{bubeck16testing} was a particular case of Theorem \ref{thm:sebtheorem}, which gave a sharp threshold for detection in the total variation distance. 
The Gaussian setting was studied further in \cite{miklos19smooth}, where a smooth transition between the regimes $n^3 \gg d$ and $n^3 \ll d$, was shown to hold. The proof of such results was facilitated by the fact that $\widetilde{W}_{n,d}^2(\gamma)$ has a tractable density with respect to the Lebesgue measure. This is not the case in general though.\\

In a follow-up (\cite{bubeck19entropic}), as discussed above, the results of \cite{bubeck16testing} were expanded to the relative entropy distance and to Wishart tensors $\widetilde{W}_{n,d}^2(\mu)$, where $\mu$ is a log-concave product measure. Specifically, it was shown that one may consider relative entropy in the formulation of Theorem \ref{thm:sebtheorem}, and that
$$\mathrm{Ent}\left(\widetilde{W}_{n,d}^2(\mu)||\gamma\right)\leq C\frac{n^3\log(d)^2}{d},$$
for a universal constant $C>0$. The main idea of the proof was a clever use of the chain rule for relative entropy along with ideas adapted from the one-dimensional entropic central limit theorem proven in \cite{artstein04rate}.  We do note the this result is not directly comparable to our results. As remarked, our results hold in Stein's discrepancy. In general, Stein's discrepancy and relative entropy are not comparable. However, one may bound the relative entropy by the discrepancy, in some cases. One such case, is when the measure has a finite Fisher information. $\widetilde{W}_{n,d}^2(\gamma)$ is an example of such a measure.\\

The question of handling dependencies between the entries of the matrix $\mathbb{X}$ in \eqref{eq: wishart representation} was also tackled in \cite{nourdin2018asymptotic}. The authors considered the case where the rows of $\mathbb{X}$ are i.i.d. copies of a Gaussian measure whose covariance is a symmetric Toeplitz matrix. The paper employed Stein's method in a clever way, which seems to be somewhat different from our approach. \\

For another direction of handling dependencies, note that if the rows of $\mathbb{X}$ are independent, but not isotropic, Gaussian vectors, then by applying an orthogonal transformation to the rows we can obtain a matrix with independent entries which have different variances. Such measures were studied in \cite{eldan2016information}. Specifically if $\alpha = \{\alpha_i\}_{i=1}^d \subset \RR^+$, with $\sum\alpha_i^2 = 1$ and $X_i \sim \gamma$ are independent, then the paper introduced $W_{n,\alpha}^2(\gamma),$ as the law of,
$$\sum\alpha_i\left(X_i^{\stensor 2} - \EE\left[X_i^{\stensor 2}\right]\right).$$ The following variant of Theorem \ref{thm:sebtheorem} was given:
\begin{equation} \label{eq:nonhomog}
\mathrm{Ent}\left(\widetilde{W}_{n,\alpha}^2(\gamma)||\gamma_{\binom{n}{2}}\right)\leq C n^3\sum\alpha_i^4.
\end{equation}
When $\alpha_i \equiv \frac{1}{\sqrt{d}}$, this recovers the previous known result. We mention here that our method applies to non-homogeneous sums as well, with the same dependence on $\alpha$. See Section \ref{sec:nonhom} for a comparison with the above result, as well as the one in \cite{nourdin2018asymptotic}.

The authors of \cite{nourdin2018asymptotic} also dealt with Wishart tensors, when the underlying measure is the standard Gaussian. It was shown that for some constant $C_p$, which depends only on $p$,
$$\mathcal{W}_1\left(\widetilde{W}_{n,d}^p(\gamma),\gamma_{\binom{n}{p}}\right) \leq C_p\sqrt{\frac{n^{2p-1}}{d}}.$$
Thus, our results should also be seen as a direct generalization of this bound.\\

Wishart tensors have recently gained interest in the machine learning community (see \cite{anandkumar14tensor,shi2018anomaly} for recent results and applications). To mention a few examples: In \cite{jiang2019limiting} the distribution of the maximal entry of $\widetilde{W}_{n,d}^p(\mu)$ is investigated. Using tools of random matrix theory, the spectrum of Wishart tensors is analyzed in \cite{ambianis12random}, while \cite{lytova18central} studies the central limit theorem for spectral linear statistics of $W_{n,d}^p(\mu)$. Results of a different flavor are given in \cite{vershynin2019concentration}, where exponential concentration is studied for a class of random tensors.
\subsection{Organization}
The rest of this paper is organized in the following way: In Section \ref{sec:perlim} we introduce some preliminaries from Stein's method and concentration of measure, which will be used in our proofs. In Section \ref{sec:method} we describe our method and present the necessary ideas with which we will prove our results. In particular, we will state Theorem \ref{thm: main}, which will act as our main technical tool. In Section \ref{sec:kernel} we introduce a construction in Stein's theory which will be used in Section \ref{sec:proofmain} to prove Theorem \ref{thm: main}. Sections \ref{sec: unconditional}, \ref{sec: product} and \ref{sec: uncondtional extension} are then devoted to the proofs of Theorems \ref{thm: uniform measures}, \ref{thm: spec measures} and \ref{thm: p2} respectively. Finally, in Section \ref{sec:nonhom} we discuss a generalization of our results to non-homogeneous sums of the tensor powers.

\paragraph{Acknowledgments}
We are grateful to Ronen Eldan and Max Fathi for many useful comments and suggestions and to
S\'ebastien Bubeck, Shirshendu Ganguly and Ivan Nourdin for reading a preliminary version of this paper. We would also like to thank the two anonymous referees for thoroughly reading this paper. Their insights have greatly improved the overall readability and presentation of the paper.
\section{Preliminaries} \label{sec:perlim}
In this section we will describe our method and explain how to derive the stated results. We begin with some preliminaries on Stein's method.
\subsection{Stein kernels}
Denote $\mathcal{M}_n(\RR)$ to be the space of $n \times n$ matrices with real entries, endowed with the Hilbert-Schmidt inner product,
$$\langle A, B\rangle_{HS} = \mathrm{Tr}\left(AB^T\right).$$
For a measure $\mu$ on $\RR^n$, we say that a measurable matrix valued map $\tau: \RR^n \to \mathcal{M}_n(\RR)$ is a Stein kernel for $\mu$, if the following equality holds, for all differentiable test functions $f:\RR^n \to \RR^n$,
$$\int\limits \langle x,f(x)\rangle d\mu(x) = \int\langle\tau(x),D f(x)\rangle_{HS}d\mu(x),$$
where $Df$ stands for the Jacobian matrix of $f$.
According to Stein's lemma, (\cite[Lemma 2.1]{chen11normal}) $\mu = \gamma$ if and only if the function $\tau(x) \equiv \mathrm{Id}$ is a Stein kernel for $\mu$. Thus, in some sense, the deviation of $\tau$ from the identity measures the distance of $\mu$ from the standard Gaussian. Led by this idea we define the Stein discrepancy to the normal distribution as,
$$S(\mu) := \inf\limits_{\tau}\sqrt{\int\norm{\tau(x) -\mathrm{Id}}_{HS}^2d\mu(x)},$$
where the infimum is taken over all Stein kernels of $\mu$. By the above, $S(\mu) = 0$ if and only if $\mu = \gamma$. Thus, while not strictly being a metric, the Stein discrepancy still serves as some notion of distance to the standard Gaussian. If $X$ is a random vector in $\RR^n$, by a slight abuse of notation we will also write $S(X)$ for $S(\mathrm{Law}(X))$. \\

Stein kernels exhibit several nice properties which make their analysis tractable for normal approximations. Let $X \sim \mu$, it is straightforward to verify that if $\tau_X$ is a Stein kernel of $X$ and $A$ is a linear operator with compatible dimensions, a Stein kernel for $AX$ is given by: 
\begin{align} \label{eq: linear transport}
\tau_{AX}(x) = A\EE\left[\tau(X)|AX = x\right]A^T
\end{align}
(see \cite[Section 3]{courtade19exsitence} for some examples). Also, it is not hard to see that, if $\{X_i\}_{i=1}^d$ are i.i.d. as $\mu$ and ${\bf{X}} = (X_1,...,X_d)$, then a Stein kernel for $\bf{X}$ may be realized as an $nd \times nd$ block diagonal matrix, with each main-diagonal block taken as $\tau_X$.
By combining the above constructions we then get that, if $S_d = \frac{1}{\sqrt{d}}\sum\limits_{i=1}^dX_i$, then
\begin{equation} \label{eq: normalized stein sums}
\tau_{S_d}(x) = \frac{1}{d}\sum\limits_{i=1}^d\EE\left[\tau_{X}(X_i)|S_d = x\right],
\end{equation}
is a Stein kernel for $S_d$. 

Now, assume that $X$ is isotropic. By choosing the test function $f$ to be linear in the definition of the Stein kernel we may see that 
\begin{equation} \label{eq: kernel moment}
\EE\left[\tau_X(X)\right] = \mathrm{Cov}(X) =  \mathrm{Id}.
\end{equation}
Thus, $\left(\tau_X(X_i) - \mathrm{Id}\right)$ is a centered random variable, and the above observations show,
\begin{align*} 
S^2(S_d)&\leq \EE\left[\norm{\tau_{S_d}(S_d)-\mathrm{Id}}_{HS}^2\right] = \EE\left[\norm{\frac{1}{d}\sum\limits_{i=1}^d\EE\left[\tau_X(X_i)-\mathrm{Id}|S_d\right]}_{HS}^2\right]\nonumber\\
&\leq \frac{1}{d^2}\EE\left[\norm{\sum\limits_{i=1}^d\tau_X(X_i)-\mathrm{Id}}_{HS}^2\right] = \frac{1}{d}\EE\left[\norm{\tau_X(X)-\mathrm{Id}}_{HS}^2\right],
\end{align*}
where in the second inequality we have applied Jensen's inequality and where we have used the fact that $(\tau_X(X_i)-\mathrm{Id})$ are i.i.d. and centered for the last equality.
By taking the infimum over all Stein kernels, we get
\begin{equation} \label{eq: CLT stein}
S^2(S_d) \leq \frac{S^2(X)}{d}.
\end{equation}

\subsubsection{Stein's discrepancy as a distance}
In this section we discuss the relations between Stein's discrepancy to other, more classical notions of distance. This notion has recently gained prominence in the study of convergence rates along the high-dimensional central limit theorem (see \cite{courtade19exsitence,fathi19stein,nourdin14entropy} for several examples as well as the book \cite{nourdin2012normal}). This popularity stems, among others, from the fact that the Stein's discrepancy controls several other known distances.\\

First, by the Kantorovich-Rubinstein duality for $\mathcal{W}_1$ (\cite{villani2008optimal}) it is not hard to show that $\mathcal{W}_1(\mu,\gamma) \leq S(\mu)$ (the interested reader may see \cite[Theorem 5.1.3]{nourdin2012normal}, for an idea of the proof). A more remarkable fact is that the same also holds for the quadratic Wasserstein distance (\cite[Proposition 3.1]{ledoux15stein}),
\begin{equation} \label{eq:steintowass}
\mathcal{W}_2^2\left(\mu,\gamma\right)\leq S^2(\mu).
\end{equation}
In fact, by slightly changing the definition of the discrepancy, we may bound the Wasserstein distance of any order. This is the content of Proposition 3.1 in \cite{fathi19stein} which shows that if $\tau$ is a Stein kernel of $\mu$, then,
$$\mathcal{W}_m(\mu,\gamma) \leq C_m \sqrt[m]{\int\norm{\tau(x) -\mathrm{Id}}_{HS}^md\mu(x)},$$
where $C_m >0$ depends only on $m$.
Our methods may then be adapted to deal with the Wasserstein distance of any order. \\
In some cases, one may also compare relative entropy to Stein's discrepancy, which is the content of the so-called HSI inequality from \cite{ledoux15stein}. According to the inequality,
$$\mathrm{Ent}(\mu||\gamma) \leq \frac{1}{2}S^2(\mu)\log\left(1 + \frac{\mathrm{I}(\mu||\gamma)}{S^2(\mu)}\right),$$
where $\mathrm{I}(\mu||\gamma)$ is the (relative) Fisher information of $\mu$,
$$\mathrm{I}(\mu||\gamma) = \int\norm{\nabla\log\left(\frac{d\mu}{d\gamma}\right)}_2^2d\mu.$$
So, if we can show that for some large $d$, $\widetilde{W}_{n,d}^p(\mu)$ has a finite Fisher information, we could expand our results, and give bounds in relative entropy. Unfortunately, verifying that a measure has finite Fisher information is a non-trivial task in high-dimensions (see Section 5 in \cite{ledoux15stein} for further discussion).
\subsection{Smooth measures and concentration inequalities} \label{subsec:smooth}
Our result will mostly apply for measures which satisfy some regularity conditions. We detail here the main properties which will be relevant. \\

A measure $\mu$ is said to satisfy the $L_1$-Poincar\'e inequality with constant $c$ if, for any differentiable function $f$ with $0$-median,
$$\int\left|f\right|d\mu \leq \frac{1}{c}\int\norm{\nabla f}_2d\mu.$$
Remark that the $L_1$-Poincar\'e inequality is equivalent, up to constants, to the Cheeger's isopermetric inequality. That is, if $\mu$ satisfies the $L_1$-Poincar\'e inequality with constant $c>0$, then for some other constant $c'>0$, depending only on $c$, and for every measurable set $B$,
$$\mu^+\left(\partial B\right)\geq c'\mu(B)\left(1-\mu(B)\right).$$
where $\mu^+\left(\partial B\right)$ is the outer boundary measure of $B$. Moreover, up to universal constants, the $L_1$-Poincar\'e inequality implies an $L_2$-Poincar\'e inequality. We refer the reader to \cite{buser82isoperimetric} for further discussion of those facts.\\

For a given measure, the above conditions imply the existence sub-exponential tails. In particular, if $\mu$ is a centered measure which satisfies the $L_1$-Poincar\'e inequality (or $L^2$) with constant $c$, then, for any $m \geq 2$:
\begin{equation} \label{eq: sub-exponential}
\EE\left[\norm{X}_2^m\right]\leq C_{m}\left(\frac{1}{c}\right)^{\frac{m}{2}}\EE\left[\norm{X}_2^2\right]^{\frac{m}{2}},
\end{equation}
where $C_m$ depends only on $m$ (see \cite{milman2009convexity} for the connection between Poincar\'e inequalities and exponential concentration).
All log-concave measures satisfy a Poincar\'e inequality, which implies that they have sub-exponential tails. In fact, a stronger statement holds for log-concave measures, and one may omit the dependence on the Poincar\'e constant in \eqref{eq: sub-exponential} (see \cite[Theorem 5.22]{lovasz07geometry}). Thus, if $X$ is a log-concave random vector,
$$\EE\left[\norm{X}_2^m\right]\leq C'_{m}\EE\left[\norm{X}_2^2\right]^{\frac{m}{2}}.$$
for some constant $C_m' > 0$, depending only on $m$. 
\section{The method} \label{sec:method}
With the above results, the following theorem is our main tool, with which we may prove CLTs for $W_{n,d}^p(\mu)$.
\begin{theorem} \label{thm: main}
	Let $X\sim \mu$ be an isotropic random vector in $\RR^n$ and let $G \sim \mathcal{N}(0,\mathrm{Id})$ stand for the standard Gaussian. Assume that $ X \stackrel{\mathrm{law}}{=} \vphi(G)$ for some $\vphi:\RR^n \to \RR^n$, which is locally-Lipschitz and let $A:\mathrm{Sym}^p(\RR^n) \to V$ be a linear transformation with $V \subset \mathrm{Sym}^p(\RR^n)$, such that $A_*W_{n,d}^p(\mu)$ is isotropic.
	Then, for any $2 \leq p \in \NN$, 
	$$S^2\left(A_*W_{n,d}^p(\mu)\right) \leq 2\norm{A}_{op}^2p^4\cdot\frac{n}{d}\sqrt{\EE\left[\norm{X}_2^{8(p-1)}\right]}\sqrt{\EE\left[\norm{D\vphi(G)}_{op}^8\right]} +\frac{2n^p}{d}.$$
\end{theorem}
Some remarks are in order concerning the theorem. We first discuss the role of the matrix $A$. Recall that, in order to use the sub-additive property \eqref{eq: CLT stein} of the Stein discrepancy, the random vectors need to be isotropic. This can be achieved via a normalizing linear transformation. However, the term $\|A\|_{op}$ which appears in the theorem tells us that if the covariance matrix of $W_{n,d}^p(\mu)$ has very small eigenvalues, the normalizing transformation might have adverse effects on the rate of convergence. To avoid this, we will sometimes project the vectors into a subspace, such as $\widetilde{\mathrm{Sym}}^p(\RR^n)$, where the covariance matrix is easier to control. Thus, $A$ should be thought of as a product of a projection matrix with the inverse of a covariance matrix on the projected space. For our applications we will make sure that,
$\norm{A}_{op} = O(1)$.\\

Concerning the other terms in the stated bound, there are two terms which we will need to control, $\sqrt{\EE\left[\norm{X}_2^{8(p-1)}\right]}$ and $\sqrt{\EE\left[\norm{D\vphi(G)}_{op}^8\right]}$. Since we are mainly interested in measures with sub-exponential tails, the first term will be of order $n^{2p-2}$ and we will focus on the second term.
Thus, in some sense, our bounds are meaningful mainly for measures which can be transported from the standard Gaussian with low distortion. Still, the class of measures which can be realized in such a way is rather large and contains many interesting examples.\\

A map $\psi$ is said to transport $G$ to $X$ if $\psi(G)$ has the same law as $X$. To apply the result we must realize $X$ by choosing an appropriate transport map. It is a classical fact (\cite{brenier87decomposition}) that, whenever $\mu$ has a finite second moment and is absolutely continuous with respect to $\gamma$, there is a distinguished map which transports $G$ to $X$. Namely, the Brenier map which minimizes the quadratic distance,
$$\vphi_\mu := \inf\limits_{\psi: \psi(G) \stackrel{\mathrm{law}}{=} X} \EE\left[\norm{G - \psi(G)}_2^2\right].$$

The Brenier map has been studied extensively (see \cite{caffarelli90localiztion,caffarelli90regularity, colombo2019bounds,kolesnikov10global} for example). Here, we will concern ourselves with cases where one can bound the derivative of $\vphi_\mu$. The celebrated Caffareli's log-concave perturbation theorem (\cite{caffarelli00monotoncity}) states that if $\mu$ is $L$-uniformly log-concave, then $\vphi_\mu$ is $\frac{1}{L}$-Lipschitz. In particular, $\vphi$ is differentiable almost everywhere with
$$\norm{D\vphi_\mu(x)}_{op} \leq \frac{1}{L}.$$
In this case we get
\begin{equation} \label{eq: uniformly caffareli}
\sqrt{\EE\left[\norm{D\vphi_\mu(G)}_{op}^8\right]} \leq \frac{1}{L^4}.
\end{equation}
Theorem \ref{thm: uniform measures} will follow from this bound. The reason why the theorem specializes to unconditional measures is that, in light of the dependence on the matrix $A$ in Theorem \ref{thm: main}, we need to have some control over the covariance structure of $\widetilde{W}_{n,d}^p(\mu)$. It turns out, that for unconditional log-concave measures the covariance of $\widetilde{W}_{n,d}^p(\mu)$ is well behaved.
The result might be extended to uniformly log-concave measures which are not necessarily unconditional as long as we allow the bound to depend on the minimal eigenvalue of the covariance matrix of  $\widetilde{W}_{n,d}^p(\mu)$.\\

There are more examples of measures for which the Brenier map admits bounds on the Lipschitz constant. In \cite{colombo17lipschitz} it is shown that for measures $\mu$ which are bounded perturbation of the Gaussian, including radially symmetric measures, $\vphi_{\mu}$ is Lipschitz. The theorem may thus be applied to those measures as well. \\

One may also consider cases where the transport map is only locally-Lipschitz in a well behaved way. For example, consider the case where $X = (X_{(1)},...,X_{(n)})\sim\mu$ is a product measure. That is, for $i \neq j$, $X_{(i)}$ is independent from $X_{(j)}$. Suppose that for $i = 1,\dots,n$, there exist functions $\vphi_i:\RR\to\RR$ such that, if $G^1$ is a standard Gaussian in $\RR$, then $\vphi_i(G^1) \stackrel{\mathrm{law}}{=}X_{(i)}$ and that $\vphi$ has polynomial growth. Meaning, that for some constants $\alpha,\beta \geq0$,
$$\vphi'_i(x) \leq \alpha(1+|x|^\beta).$$  
Since $\mu$ is a product measure, it follows that the map $\vphi = (\vphi_1,\dots,\vphi_n)$ transports $G$ to $X$ and that,
$$\norm{D\vphi(x)}_{op} \leq \alpha(1+\norm{x}_\infty^\beta).$$
Thus, for product measures, we can translate bounds on the derivative of one-dimensional transport maps into multivariate bounds involving the $L_\infty$ norm. Theorem \ref{thm: spec measures} will be proved by using these ideas and known estimates on the one-dimensional Brenier map (also known as monotone rearrangement). Results like Theorem \ref{thm: p2} can then be proven by bounding the covariance matrix of $W^2_{n,d}(\mu)$. Indeed, this is the main ingredient in the proof of the theorem.\\

One may hope that Theorem \ref{thm: main} could be applied to general log-concave measures. However, this would be a highly non-trivial task. Indeed, if we wish to use Theorem \ref{thm: main} in order to verify the threshold $n^{2p-1} \ll d$, up to logarithmic terms, we should require that for any isotropic log-concave measure $\mu$, there exists a map $\psi_\mu$ such that $\psi_\mu(G) \sim \mu$ and $\EE\left[\norm{D\psi_\mu(G)}_{op}^8\right] \leq \log(n)^\beta$, for some fixed $\beta \geq 0$.  Then, by applying the Gaussian $L_2$-Poincar\'e inequality to the function $\norm{\cdot}_2$, we would get,
\begin{align*}
\mathrm{Var}\left(\norm{\psi_\mu(G)}_2\right) &\leq \EE\left[\norm{D\left(\norm{\psi_\mu(G)}_{2}\right)}^2_2\right]\\
&= \EE\left[\norm{\frac{\psi_\mu(G)}{\norm{\psi_\mu(G)}_{2}}D\psi_\mu(G)}^2_2\right]\\
&\leq \EE\left[\norm{\frac{\psi_\mu(G)}{\norm{\psi_\mu(G)}_{2}}}_2^2\cdot\norm{D\psi_\mu(G)}_{op}^2\right]\\
&=\EE\left[\norm{D\psi_\mu(G)}_{op}^2\right] \leq \log(n)^{\frac{\beta}{4}},
\end{align*}
where the first equality is the chain rule and the second inequality is a consequence of considering $D\psi_\mu$ as an $n \times n$ matrix.
This bound would, up to logarithmic factors, verify the \emph{thin-shell} conjecture (see \cite{anttila00central}), and, through the results of \cite{eldan13thin}, also the \emph{KLS} conjecture (\cite{kannan95isoperimetric}). These two conjectures are both famous long-standing open problems in convex geometry. Thus, while we believe that similar results should hold for general log-concave, it seems such claims would require new ideas.\\

Another evidence for the possible difficulty of determining optimal convergence rates for general log-concave vectors can be seen from the case $p=1$. In the standard setting of the CLT, the best known convergence rates (\cite{courtade19exsitence,eldan2020CLT}), in quadratic Wasserstein distance, depend on the Poincar\'e constant of the isotropic log-concave measure. Bounding the Poincar\'e constant is precisely the object of the KLS conjecture. So, proving a convergence rate which does not depend on the specific log-concave measure seems to be intimately connected with the conjecture. This suggests the question might be a genuinely challenging one. 
\subsection{High-level idea}
We now present the idea behind the proof of Theorem \ref{thm: main} and detail the main steps. We first provide an informal explanation of why standard techniques fail to give optimal bounds. We may treat $W_{n,d}^p(\mu)$ as a sum of independent random vectors and invoke Theorem 7 from \cite{chatterjee08multivariate} (similar results will encounter the same difficulty). So, if $X \sim \mu$, optimistically, the theorem will give,
$$\mathcal{W}_1(W_{n,d}^p(\mu), \gamma) \leq \frac{\EE\left[\norm{X^{\stensor p}}^3_2\right]}{\sqrt{d}},$$
where we take the Euclidean norm of $X^{\stensor p}$ when considered as a vector in $\mathrm{Sym}^p(\RR^n)$. Since $\dim\left(\mathrm{Sym}^p(\RR^n)\right) \simeq n^p$, and we expect each coordinate of $X^{\stensor p}$ to have magnitude, roughly $O(1)$, Jensen's inequality gives:
$$\EE\left[\norm{X^{\stensor p}}^3_2\right] \geq \EE\left[\norm{X^{\stensor p}}^2_2\right]^{\frac{3}{2}}  \gtrsim\dim\left(\mathrm{Sym}^p(\RR^n)\right)^{\frac{3}{2}} \gtrsim n^{\frac{3p}{2}}.$$
This is worse than the bound $\sqrt{n^{2p-1}}$, achieved by Theorem \ref{thm: main}.\\

The high-level plan of our proof is to use the fact that $X^{\stensor p}$ has some low-dimensional structure. We will construct a map which transports the standard Gaussian $G$, from the lower dimensional space $\RR^n$ into the law of $X^{\stensor p}$ in the higher dimensional space $\mathrm{Sym}^p(\RR^n)$. In some sense, the role of this transport map is to preserve the low-dimensional randomness coming from $\RR^n$. The map can be constructed in two steps, first use a transport map $\vphi$, such that $\vphi(G) \stackrel{\mathrm{law}}{=} X$, and then take its tensor power $\vphi(G)^{\stensor p}$. We will use this map in order to construct a Stein kernel and show that tame tails of the map's derivative translate into small norms for the Stein kernel.

\section{From transport maps to Stein kernels} \label{sec:kernel}
We now explain how to construct a Stein kernel from a given transport map. For the rest of this section let $\nu$ be a measure on $\RR^N$ and $Y \sim \nu$. Recall the definition of a Stein kernel; A matrix-valued map, $\tau:\RR^N \to \mathcal{M}_N(\RR)$, is a Stein kernel for $\nu$, if for every smooth $f:\RR^N \to \RR^N$, 
$$\EE\left[\langle Y, f(Y) \rangle\right] = \EE\left[\langle\tau(Y),Df(Y)\rangle_{HS}\right].$$
Our construction is based on differential operators which arise naturally when preforming analysis in Gaussian spaces. We incorporate into this construction the idea of considering transport measures between spaces of different dimensions. For completeness, we give all of the necessary details, but see \cite{nourdin2012normal,huang00introduction} for a rigorous treatment. 
 
\subsection{Analysis in finite dimensional Gauss space}
We let $\gamma$ stand for the standard Gaussian measure in $\RR^N$ and consider the Sobolev subspace of weakly differentiable  functions,
$$W^{1,2}(\gamma) := \{f \in L^2(\gamma)| f \text{ is weakly differntiable, and } \EE_{\gamma}\norm{Df}^2_2 < \infty\}.$$
where $D:W^{1,2}(\gamma)\to L^2(\gamma,\RR^N)$ is the natural (weak) derivative operator. We will mainly care about the fact that locally-Lipschitz functions are weakly differentiable the reader is referred to the second chapter of \cite{ziemer89weakly} for the necessary background on Sobolev spaces. 

The divergence $\delta$ is defined to be the formal adjoint of $D$, so that for $g: \RR^N \to \RR^N,$
$$\EE_\gamma\left[\langle Df, g\rangle\right] = \EE_\gamma\left[f \delta g\right].$$
$\delta$ is given explicitly by the relation
$$\delta g(x) = \langle x, g(x)\rangle - \mathrm{div}(g(x)),$$
where $\mathrm{div}(g(x)) = \sum\limits_{i=1}^N \frac{\partial g_i}{\partial x_i}(x)$.

The Ornstein-Uhlenbeck (OU) operator is now defined by $L:= -\delta\circ D$.
On functions, $L$ operates as $Lf(x) = - xD f(x) + \Delta f(x)$. The operator $L$ also serves as the infinitesimal generator of the OU semi-group (\cite[Proposition 1.3.6]{nourdin14entropy}). That is,
$$L = \frac{d}{dt}P_t\Big\vert_{t=0},$$
where 
\begin{equation} \label{eq: OU sg}
P_t f(x) := \EE_{N \sim \gamma}\left[f(e^{-t}x + \sqrt{1-e^{-2t}}N)\right].
\end{equation}
The following fact, which may be proved by the Hermite decomposition of $L^2(\gamma)$, will be useful; There exists an operator, denoted $L^{-1}$ such that $LL^{-1}f = f$. In particular, on the subspace of functions whose Gaussian expectation vanishes, $L^{-1}$ is the inverse of $L$ (\cite[Proposition 2.8.11]{nourdin2012normal}).\\

We now introduce a general construction for Stein kernels. By a slight abuse of notation, even when working in different dimensions, we will refer to the above differential operators as the same, as well as extending them to act of vector and matrix valued functions. Note that, in particular, if $f$ is a vector-valued function and $g$ is matrix-valued of compatible dimensions, then,
$$\EE_\gamma\left[\langle Df,g \rangle_{HS}\right] = \EE_\gamma\left[\langle f,\delta g \rangle\right].$$
\begin{lemma} \label{lem: construction}
	Let $\gamma_m$ be the standard Gaussian measure on $\RR^m$ and let $\vphi:\RR^m\to \RR^N$ be weakly differentiable. Set $\nu = \vphi_*\gamma_m$ and suppose that $\int\limits_{\RR^N}xd\nu = 0$. Then, if the following expectation is finite for $\nu$-almost every $x \in \RR^N$,
	$$\tau_\vphi(x) := \EE_{y\sim\gamma_m}\left[(-DL^{-1})\vphi(y) (D\vphi(y))^T| \vphi(y) = x\right],$$
	is a Stein kernel of $\nu$.
\end{lemma}
\begin{proof}
	Let $f:\RR^N \to \RR^N$ be a smooth function and set $Y \sim \nu, G \sim \gamma_m$. Our goal is to show 
	$$\EE\left[\langle D f(Y),\tau_\vphi(Y)\rangle_{HS}\right] = \EE\left[\langle f(Y),Y\rangle \right].$$
	Before turning to the calculations let us make explicit the dimensions of the objects which will be involved.
	$Df$ is an $N \times N$ matrix, while $D\varphi$ is an $N \times m$ matrix. Since $DL^{-1}\varphi$ is also an $N \times m$ matrix, it holds that $\tau_\varphi$ is an $N \times N$ matrix, as required. Now,
	\begin{align*}
	\EE\left[\langle D f(Y),\tau_\vphi(Y)\rangle_{HS}\right] &= \EE\left[\langle D f(Y),\EE\left[ (-DL^{-1})\vphi(G)(D\vphi(G))^T| \vphi(G) = Y\right]\rangle_{HS}\right]  \\
	&= \EE\left[\langle D f(\vphi(G))D\vphi(G), (-DL^{-1})\vphi(G)\rangle_{HS}\right]\\
	&= \EE\left[\langle D(f\circ \vphi(G)),(-DL^{-1})\vphi(G)\rangle_{HS} \right]&\text{(Chain rule)}\\
	&= \EE\left[\langle f\circ\vphi(G), (-\delta DL^{-1})\vphi(G)\rangle\right]&(D \text{ is adjoint to } \delta)\\
	&= \EE\left[\langle f\circ \vphi(G),LL^{-1}\vphi(G)\rangle\right]&L = -\delta D\\
	&= \EE\left[\langle f\circ\vphi(G),\vphi(G)\rangle\right] &\EE[\vphi(G)]=0\\
	&= \EE\left[\langle f(Y),Y\rangle \right].&\vphi_*\gamma_m = \nu
	\end{align*}
	In the first line, the inner product is taken in the space of $N \times N$ matrices and in the next two lines, in the space of $N \times m$ matrices. 
	Also, note that in the penultimate equality the fact $\EE\left[\varphi(G)\right]$ was important for the cancellation of $LL^{-1}$.
\end{proof}

The above formula suggests that one might control the the kernel $\tau_\vphi$ by controlling the gradient of the transport map, $\vphi$. This will be the main step in proving Theorem \ref{thm: main}. The following formula from \cite[Proposition 29.3]{nourdin2012normal} will be useful:
\begin{equation} \label{eq: DL}
	-DL^{-1}\vphi = \int\limits_0^\infty e^{-t}P_tD\vphi dt.
\end{equation}
We thus have the corollary:
\begin{corollary} \label{corr: construction}
	With the same notations as in Lemma \ref{lem: construction},
		$$\tau_{\vphi}(x) = \int\limits_0^\infty e^{-t}\EE_{y \sim \gamma_m}\left[P_tD\vphi(y)\left(D\vphi(y)\right)^T|\vphi(y) = x\right]dt.$$
\end{corollary}
We remark that the construction is based on ideas which have appeared implicitly in the literature, at least as far as \cite{chatterjee2009fluctuations} (see\cite{nourdin2012normal} for a more modern point of view). Our main novelty lies in interpreting the transport map, used in the construction, as an embedding from a low-dimensional space. Other constructions of Stein kernels use different ideas, such as Malliavin calculus (\cite{nourdin14entropy}), other notions of transport problems (\cite{fathi19stein}) or calculus of variation (\cite{courtade2018quantitative}). However, as will become clear in the next section, our construction seems particularly well adapted to the current problem, since it is well behaved with respect to compositions. That is, if $\psi, \varphi$ are two compatible maps and $\tau_{\vphi}$ is 'close' to the identity, then as long $\psi$ is not too wild, the same can be said about $\tau_{\psi \circ \vphi}$. In our setting, one should think about $\varphi$ as the transport map and $\psi(v):= v^{\otimes p}$. It is an interesting question whether other constructions for Stein kernels could be used in a similar way.\\

As a warm up we present a simple case in which we can show that the Stein kernel obtained from the construction is bounded almost surely.
\begin{lemma} \label{lem: uniform bound}
	Let the notations of Lemma \ref{lem: construction} prevail and suppose that $\norm{D\vphi(x)}_{op} \leq 1$ almost surely. Then
	$$\norm{\tau_\vphi(x)}_{op}\leq 1,$$
	almost surely. 
\end{lemma}
\begin{proof}
	From the representation \eqref{eq: OU sg} and by Jensen's inequality, $P_t$ is a contraction. That is, for any function $h$,
	\begin{equation} \label{eq: contractive}
	\EE_{y \sim \gamma_m}\left[P_t(h(y))h(y)\right] \leq \EE_{y \sim \gamma_m}\left[h(y)^2\right].
	\end{equation}
	 So, from Corollary \ref{corr: construction} and since $\norm{D\vphi(x)}_{op} \leq 1$, we get,
	\begin{align*}
		\norm{\tau_\vphi(x)}_{op} &\leq \int\limits_0^\infty e^{-t}\EE_{y \sim \gamma_m}\left[\norm{P_tD\vphi(y)\left(D\vphi(y)\right)}_{op}|\vphi(y) = x\right]dt\\
		&\leq \int\limits_0^\infty e^{-t}\EE_{y \sim \gamma_m}\left[P_t\left(\norm{D\vphi(y)}_{op}\right)\norm{D\vphi(y)}_{op}|\vphi(y) = x\right]dt\\
		&\leq \int\limits_0^\infty e^{-t}\EE_{y \sim \gamma_m}\left[\norm{D\vphi(y)}_{op}^2|\vphi(y) = x\right]dt \leq 1.
	\end{align*}
	The first inequality follows from  Jensen's inequality. The second inequality uses the fact that the matrix norm is sub-multiplicative combined with Jensen's inequality for $P_t$. The last inequality is the contractive property \eqref{eq: contractive} and the a-priori bound on $\|D\varphi\|_{op}$.
\end{proof}
\section{Proof of Theorem \ref{thm: main}} \label{sec:proofmain}
Let $A:\left(\RR^n\right)^{\otimes p} \to V$ be any linear transformation such that $A\left(X^{\otimes p} - \EE\left[X^{\otimes p}\right]\right)$ is isotropic, and let $\tau$ be a Stein kernel for $X^{\otimes p} - \EE\left[X^{\otimes p}\right]$. In light of \eqref{eq: linear transport}, we know that
\begin{align*}
S^2\left(A\left(X^{\otimes p} - \EE\left[X^{\otimes p}\right]\right)\right)&\leq \EE\left[\norm{A\tau\left(X^{\otimes p} - \EE\left[X^{\otimes p}\right]\right)A^T - \mathrm{Id}}^2_{HS}\right]\\
&\leq 2\EE\left[\norm{A\tau\left(X^{\otimes p} - \EE\left[X^{\otimes p}\right]\right)A^T}_{HS}^2\right] + 2\norm{\mathrm{Id}}_{HS}^2\\
 &\leq 2\norm{A}_{op}^2\EE\left[\norm{\tau\left(X^{\otimes p} - \EE\left[X^{\otimes p}\right]\right)}^2_{HS}\right] + 2\dim(V).
\end{align*}
 Thus, by combining the above with \eqref{eq: CLT stein}, Theorem \ref{thm: main} is directly implied by the following lemma. 
\begin{lemma} \label{lem: main lemma}
	Let $X$ be an isotropic random vector in $\RR^n$ and let $\vphi:\RR^n\to \RR^n$ be differentiable almost everywhere, such that $\vphi(G) \stackrel{\mathrm{law}}{=} X$, where $G$ is the standard Gaussian in $\RR^n$.
	Then, for any integer $p \geq 2$, there exists a Stein kernel $\tau$ of $X^{\otimes p} - \EE\left[X^{\otimes p}\right]$, such that
	$$\EE\left[\norm{\tau\left(X^{\otimes p} - \EE\left[X^{\otimes p}\right]\right)}^2_{HS}\right] \leq p^4n\sqrt{\EE\left[\norm{X}_2^{8(p-1)}\right]}\sqrt{\EE\left[\norm{D\vphi(G)}_{op}^{8}\right]}.$$
\end{lemma}
\begin{proof}
 Consider the map $u \to \vphi(u)^{\otimes p} - \EE\left[X^{\otimes p}\right]$ , which transports $G$ to $X^{\otimes p} - \EE\left[X^{\otimes p}\right]$. 
 For a vector $v \in \RR^n$ we will denote $\tilde{v}^{\otimes p} := v^{\otimes p} - \EE\left[X^{\otimes p}\right]$. Corollary \ref{corr: construction} shows that the function defined by,
	\begin{align*}
	\tau(\tilde{v}^{\otimes p}) :&= \int\limits_0^\infty e^{-t} \EE\left[P_t\left(D(\vphi(G)^{\otimes p})\right)\cdot D(\vphi(G)^{\otimes p})^T|\vphi(G)^{\otimes p} = v^{\otimes p}\right]dt,
	\end{align*}
	and which vanishes on tensors which are not of the form $\tilde{v}^{\otimes p}$, is a Stein kernel for $X^{\otimes p} - \EE\left[X^{\otimes p}\right]$.
	Note that for any two matrices $A,B,$
	$$\norm{AB}_{HS} \leq \sqrt{\mathrm{rank}(A)}\norm{A}_{op}\norm{B}_{op}.$$
	Thus, by applying Jensen's inequality several times, both for the integrals and for $P_t$, we have the bound
	\begin{align*}
	\EE\left[\norm{\tau(X^{\otimes p})}_{HS}^2\right]& \leq \int\limits_0^\infty e^{-t}\EE\left[\norm{P_t\left(D(\vphi(G)^{\otimes p})\right)\cdot D(\vphi(G)^{\otimes p})^T}_{HS}^2\right]dt\\
	&\leq \int\limits_0^\infty e^{-t}\EE\left[\mathrm{rank}(D(\vphi(G)^{\otimes p}))\norm{D(\vphi(G)^{\otimes p})}_{op}^2P_t\left(\norm{D(\vphi(G)^{\otimes p})}_{op}^2\right)\right]dt\\
	&\leq \int\limits_0^\infty n e^{-t}\EE\left[\norm{D(\vphi(G)^{\otimes p})}_{op}^2P_t\left(\norm{D(\vphi(G)^{\otimes p})}_{op}^2\right)\right]dt.
	\end{align*}
	To see the why the last inequality is true, observe that $v \to \vphi(v)^{\otimes p}$ is a map from $\RR^n$ to $\RR^{n^p}$, hence $D(\vphi(G)^{\otimes p})$ is an $n^p \times n$ matrix, which leads to $\mathrm{rank}(D(\vphi(G)^{\otimes p})) \leq n$.
	We now use the fact that $P_t$ is a contraction, as in \eqref{eq: contractive}, so that for every $t >0$,
	$$\EE\left[\norm{D(\vphi(G)^{\otimes p})}_{op}^2P_t\left(\norm{D(\vphi(G)^{\otimes p})}_{op}^2\right)\right] \leq \EE\left[\norm{D(\vphi(G)^{\otimes p})}_{op}^4\right].$$
	So,
	$$\EE\left[\norm{\tau(X^{\otimes p})}_{HS}^2\right]\leq n\EE\left[\norm{D(\vphi(G)^{\otimes p})}_{op}^4\right].$$
	We may realize the map $v \to \vphi(v)^{\otimes p}$ as the $p$-fold Kronecker power (the reader is referred to \cite{petersen12matrix} for the relevant details concerning the Kronecker product) of $\vphi(v)$.
	With $\otimes$ now standing for the Kronecker product, the following Leibniz law holds for the Jacobian:
	$$D\left(\vphi(x)^{\otimes p}\right) = \sum\limits_{i=1}^p \vphi(x)^{\otimes i - 1}\otimes D\vphi(x)\otimes \vphi(x)^{\otimes p-i}.$$
	The Kronecker product is multiplicative with respect to singular values, and for any $A_1,...,A_p$ matrices,
	$$\norm{A_1\otimes...\otimes A_p}_{op} = \prod\limits_{i=1}^p\norm{A_i}_{op}.$$
	We then have,
\begin{align*}
	\EE\left[\norm{\tau(X^{\otimes p})}_{HS}^2\right] &\leq n\EE\left[\norm{D(\vphi(G)^{\otimes p})}_{op}^4\right]\\
	&= n\EE\left[\norm{\sum\limits_{i=1}^p \vphi(G)^{\otimes i - 1}\otimes D\vphi(G)\otimes \vphi(G)^{\otimes p-i}}^4_{op}\right] \\
	&\leq n \EE\left[\left(\sum\limits_{i=1}^p \norm{\vphi(G)^{\otimes i - 1}\otimes D\vphi(G)\otimes \vphi(G)^{\otimes p-i}}_{op}\right)^4\right],
\end{align*}
where the operator norm here is considered on the space of $n^p \times n$ matrices. The multiplicative property of the Kronecker product shows that for every $i = 1,\dots p$,
$$\norm{\vphi(G)^{\otimes i - 1}\otimes D\vphi(G)\otimes \vphi(G)^{\otimes p-i}}_{op} = \norm{\vphi(G)}_2^{(p-1)}\norm{D\vphi(G)}_{op},$$
where now the operator norm is considered on the space of $n \times n$ matrices, and one can think about the Euclidean norm as the operator on the space of $1 \times n$ matrices. Thus,
\begin{align*}
  \EE\left[\norm{\tau(X^{\otimes p})}_{HS}^2\right]	&\leq np^4\EE\left[\norm{\vphi(G)}_2^{4(p-1)}\norm{D\vphi(G)}_{op}^4\right]\\
	&\leq np^4\sqrt{\EE\left[\norm{\vphi(G)}_{2}^{8(p-1)}\right]}\sqrt{\EE\left[\norm{D\vphi(G)}_{op}^{8}\right]}\\
	&=np^4\sqrt{\EE\left[\norm{X}_{2}^{8(p-1)}\right]}\sqrt{\EE\left[\norm{D\vphi(G)}_{op}^{8}\right]},
\end{align*}
where the last inequality is Cauchy-Schwartz.
\end{proof}
\section{Unconditional log-concave measures; Proof of Theorem \ref{thm: uniform measures}} \label{sec: unconditional}
We now wish to apply Theorem \ref{thm: main} to unconditional measures which are uniformly log-concave. In this case, we begin by showing that the covariance of $\widetilde{W}_{n,d}^p(\mu)$ is well conditioned.
\begin{lemma} \label{lem: uncondtional covariance}
	Let $\mu$ be an unconditional log-concave measure on $\RR^n$ and let $\widetilde{\Sigma}_p(\mu)$ denote the covariance matrix $\widetilde{W}_{n,d}^p(\mu)$. Then, there exists a constant $c_p > 0$, depending only on $p$, such that if $\tilde{\lambda}_{\mathrm{min}}$ stands for the smallest eigenvalue of $\widetilde{\Sigma}_p(\mu)$, then
	$$c_p \leq \tilde{\lambda}_{\mathrm{min}}.$$
\end{lemma}
\begin{proof}
	We write $X =(X_{(1)},...,X_{(n)})$ and observe that $\Sigma^p(\mu)$ is diagonal. Indeed, if $1 \leq j_1 < j_2<...<j_p\leq n$ and $1 \leq j'_1 < j'_2<...<j'_p\leq n$ are two different sequences of indices then the covariance between $X_{(j_1)}\cdot...\cdot X_{(j_p)}$ and $X_{(j'_1)}\cdot...\cdot X_{(j'_p)}$ can be written as
	$$\EE\left[X_{(i_1)}\cdot X_{(i_2)}^{n_2}\cdot...\cdot X_{(i_k)}^{n_k}\right],$$
	where $p+1 \leq k \leq 2p$ and for every $i = 2,...,k$, $n_i \in \{1,2\}$. By \eqref{eq: unconditional moments}, those terms vanish. Thus, in order to prove the lemma, it will suffice to show that for every set of distinct indices $j_1,...,j_p$,
	$$c_p \leq \EE\left[\left(X_{(j_1)}\cdot...\cdot X_{(j_p)}\right)^2\right],$$
	for some constant $c_p>0$, which depends only on $p$. If we consider the random isotropic and log-concave vector $(X_{j_1},...,X_{j_p})$ in $\RR^p$, the existence of such a constant is assured by the fact that the density of this vector is uniformly bounded from below on some ball around the origin (see \cite[Theorem 5.14]{lovasz07geometry}).
\end{proof}
We now prove Theorem \ref{thm: uniform measures}.
\begin{proof}[Proof of Theorem \ref{thm: uniform measures}]
	Set $P: \left(\RR^n\right)^{\otimes p} \to \widetilde{\mathrm{Sym}}^p\left(\RR^n\right)$ to be the linear projection operator and $\widetilde{\Sigma}_p(\mu)$ to be as in Lemma \ref{lem: uncondtional covariance}. Denote $A = \sqrt{\widetilde{\Sigma}_p^{-1}(\mu)}P$. Then, $A\left(X^{\otimes p}- \EE\left[X^{\otimes p}\right]\right)$ is isotropic and has the same law as $\sqrt{\widetilde{\Sigma}_p^{-1}(\mu)}_*\widetilde{W}_{n,d}^p(\mu)$. The lemma implies
	$$\norm{A}_{op}^2 \leq \frac{1}{c_p}.$$
	As $X$ is log-concave and isotropic, from \eqref{eq: sub-exponential}, we get
	$$\sqrt{\EE\left[\norm{X}_2^{8(p-1_)}\right]} \leq C_pn^{2p-2}.$$
	$X$ is also $L$-uniformly log-concave. So, as in \eqref{eq: uniformly caffareli}, if $\vphi_{\mu}$ is the Brenier map, sending the standard Gaussian $G$ to $X$,
	$$\sqrt{\EE\left[\norm{D\vphi_{\mu}(G)}_{op}^8\right]} \leq \frac{1}{L^4}.$$
	Combining the above displays with Theorem \ref{thm: main}, gives the desired result.
\end{proof}
\section{Product measures; Proof of Theorem \ref{thm: spec measures}} \label{sec: product}
As mentioned in Section \ref{sec:method}, when $\mu$ is a product measure, transport bounds on the marginals of $\mu$ may be used to construct a transport map $\vphi$ whose derivative satisfies an $L_\infty$ bound of the form,
\begin{equation} \label{eq: inf bounds}
\norm{D\vphi(x)}_{op} \leq \alpha(1+\norm{x}_\infty^\beta).
\end{equation}
for some $\alpha,\beta \geq 0$. 
Such conditions can be verified for a wide variety of product measures. For example, it holds, a fortiori, when the marginals of $\mu$ are polynomials of the standard one-dimensional Gaussian with bounded degrees. Furthermore, we mention now two more cases where the one-dimensional Brenier map is known to have tame growth. Those estimates will serve as the basis for the proof of Theorem \ref{thm: spec measures}.\\
In \cite{courtade2018quantitative} it is shown that if $\mu$ is an isotropic log-concave measure in $\RR$, and $\vphi_{\mu}$ is its associated Brenier map, then for some universal constant $C >0$,
\begin{equation} \label{eq: brenier log-concave}
\vphi'_\mu(x)\leq C(1+|x|).
\end{equation}
If, instead, $\mu$ satisfies an $L_1$-Poincar\'e inequality with constant $c_\ell>0$, then for another universal constant $C >0$
$$\vphi'_\mu(x)\leq C\frac{1}{c_\ell}(1+x^2).$$
Thus, for log-concave product measures \eqref{eq: inf bounds} holds with $\beta = 1$ and for products of measures which satisfy the $L_1$-Poincar\'e inequality it holds with $\beta = 2$. Using these bounds, Theorem \ref{thm: spec measures} becomes a consequence of the following lemma.
\begin{lemma} \label{lem: op to inf}
	Let $X$ be a random vector in $\RR^n$ and let $G$ stand for the standard Gaussian. Suppose that for some $\vphi:\RR^n \to \RR^n$, $\vphi(G) \stackrel{\mathrm{law}}{=} X$, and that $\vphi$ is differentiable almost everywhere with
	$$\norm{D\vphi(x)}_{op} \leq \alpha(1+\norm{x}_\infty^\beta),$$
	for some $\beta,\alpha >0$.
	Then, there exists a constant $C_\beta$, depending only on $\beta$, such that
	$$\EE\left[\norm{D\vphi(x)}_{op}^8\right] \leq C_\beta\alpha^8\log(n)^{4\beta}.$$
\end{lemma}
	\begin{proof}
		For any $x,y \geq 0$, the following elementary inequality holds,
		$$(x +y)^8 \leq 2^7\left(x^8 + y^8\right).$$
		Thus, we begin the proof with,
		\begin{align*}
		\EE\left[\norm{D\vphi(G)}_{op}^8\right] \leq \alpha^8\EE\left[(1+\norm{G}^\beta_{\infty})^8\right]\leq 256\alpha^8\EE\left[\norm{G}_{\infty}^{8\beta}\right].
		\end{align*}
		Observe that the density function of $\norm{G}_\infty$ is given by $n\psi\cdot \Phi^{n-1}$, where $\psi$ is the density of the standard Gaussian in $\RR$ and $\Phi$ is its cumulative distribution function. Since the product of log-concave functions is also log-concave, we deduce that $\norm{G}_\infty$ is a log-concave random variable. From \eqref{eq: sub-exponential}, we thus get
		$$\EE\left[\norm{G}_{\infty}^{8\beta}\right] \leq C'_\beta\EE\left[\norm{G}_{\infty}^2\right]^{4\beta},$$
		where $C'_\beta$ depends only on $\beta$.
		The proof is concluded by applying known estimates to $\EE\left[\norm{G}_{\infty}^2\right]$.
	\end{proof}

\begin{proof}[Proof of Theorem \ref{thm: spec measures}]
	 We first observe that, since $\mu$ is an isotropic product measure, $\widetilde{W}_{n,d}^p(\mu)$ is an isotropic random vector in $\widetilde{\mathrm{Sym}}^p(\RR^n)$. Thus, the matrix $A$ in Theorem \ref{thm: main}, reduces to a projection matrix and $\norm{A}_{op} = 1$. \\
	 Let $X \sim \mu$. For the first case, we assume that $X$ is log-concave. Since it is also isotropic, from \eqref{eq: sub-exponential} there exists a constant $C_p'$ depending only $p$, such that 
	\begin{equation} \label{eq: norm log concave}
	\sqrt{\EE\left[\norm{X}_2^{8(p-1)}\right]} \leq C_p' n^{2p-2}.
	\end{equation}
	We now let $\vphi_{\mu}$ stand for the Brenier map between the standard Gaussian $G$ and $X$. Since $X$ has independent coordinates it follows from \eqref{eq: brenier log-concave} that for some absolute constant $C > 0$.
	$$\norm{D\vphi(x)}_{op} \leq C(1+\norm{x}_{\infty}).$$
	In this case, Lemma \ref{lem: op to inf} gives:
	\begin{equation} \label{eq: transport log concave}
	\sqrt{\EE\left[\norm{D\vphi(G)}_{op}^8\right]} \leq C'\log(n)^2,
	\end{equation}
	where $C' > 0$ is some other absolute constant. Plugging these estimates into Theorem \ref{thm: main} and taking $C_p = 2C'\cdot C_p'\cdot p^4$ shows Point $1$.\\
	
	The proof of Point $2$ is almost identical and we omit it.
	For Point $3$, when each marginal of $\mu$ is a polynomial of the standard Gaussian, observe that the map $\tilde{Q}:\RR^n \to \RR^n$,
	$$\tilde{Q}(x_1,...,x_n) = (Q(x_1),...,Q(x_n)),$$
	is by definition a transport map between  $G$ and $X$. Since $Q$ is a degree $k$ polynomial, there exists some constant $C_Q$, such that
	$$Q'(x) \leq C_Q(1+|x|^{k-1}).$$
	So, from Lemma \ref{lem: op to inf}, there is some constant $C'_{Q,p}$, such that
	$$\sqrt{\EE\left[\norm{D\tilde{Q}(G)}_{op}^8\right]} \leq C'_{Q,p}\log(n)^{2(k-1)}.$$
	Moreover, using hypercontractivity (see \cite[Theorem 5.10]{janson97gaussian}), since $X$ is given by a degree $k$ polynomial of the standard Gaussian, we also have the following bound on the moments of $X$:
	$$\sqrt{\EE\left[\norm{X}_2^{8(p-1)}\right]} \leq (8p)^{2kp}\EE\left[\norm{X}_2^2\right]^{2{p-2}} = (8p)^{2kp}n^{2{p-2}}.$$
	Using the above two displays in Theorem \ref{thm: main} finishes the proof.
\end{proof}

\section{Extending Theorem \ref{thm: spec measures}; Proof of Theorem \ref{thm: p2}} \label{sec: uncondtional extension}
We now fix $X \sim \mu$ to be an unconditional isotropic log-concave measure on $\RR^n$ with independent coordinates. If $\Sigma_2(\mu)$ stands for the covariance matrix of $W_{n,d}^2(\mu)$, then, using the same arguments as in the proof of Theorem \ref{thm: spec measures}, it will be enough to show that $\Sigma_2(\mu)$ is bounded uniformly from below. Towards that, we first prove:
\begin{lemma} \label{lem: square variance}
	Let $Y$ be an isotropic log-concave random variable in $\RR$. Then
	$$\mathrm{Var}(Y^2) \geq \frac{1}{100}.$$
\end{lemma}
\begin{proof}
	Denote by $\rho$ the density of $Y$. We will use the following $3$ facts, pertaining to isotropic log-concave densities in $\RR$ (see Section 5.2 in \cite{lovasz07geometry}). 
	\begin{itemize}
		\item $\rho$ is uni-modal. That is, there exists a point $x_0 \in \RR$, such that $\rho$ is non-decreasing on $(-\infty,x_0)$ and non-increasing on $(x_0,\infty)$.
		\item $\rho(0) \geq \frac{1}{8}$ and $\rho(x) \leq 1$, for every $x \in \RR$. 
		\item $\int\limits_{|x| \geq 2}\rho(x)dx \leq \frac{1}{e}$.
	\end{itemize}
	The first observation is that either $\rho\left(\frac{1}{9}\right) \geq \frac{1}{10}$, or $\rho\left(-\frac{1}{9}\right) \geq \frac{1}{10}$. Indeed, if not, then as $\rho$ is uni-modal and $\rho(0) \geq \frac{1}{10}$,
	$$\int\limits_{-2}^2\rho(x)dx \leq \int\limits_{-\frac{1}{9}}^\frac{1}{9}\rho(x)dx + \frac{4}{10} \leq \frac{2}{9} + \frac{4}{10} < 1 - \frac{1}{e},$$
	which is a contradiction.
	We assume, without loss of generality, that $\rho\left(\frac{1}{9}\right) \geq \frac{1}{10}$. Similar considerations then show
	\begin{align*}
	\mathrm{Var}(Y^2) = \int\limits_{\RR} (x^2-1)^2\rho(x)dx \geq \frac{1}{10}\int\limits_0^\frac{1}{9}(x^2-1)^2dx \geq \frac{1}{100}.
	\end{align*}
	\iffalse
	Denote by $\rho$ the density of $Y$. Since $Y$ is log-concave, it holds that $\rho$ is uni-modal.
	We will use the following fact, known as Jacobson's inequality (see \cite{jacobson69maximum}), which restricts the maximal variance of bounded uni-modal and symmetric random variables. Thus, according to the inequality, if $k > 0$,
	$$\int\limits_{-k}^kx^2\rho(x)dx \leq \frac{4k^2}{12}.$$
	Denote $h(k) = \int\limits_{-k}^kx^2\rho(x)dx$.
	The following decomposition then holds, for every $k \geq 0$,
	\begin{align*}
	\EE\left[Y^4\right] &= \int\limits_{-\infty}^\infty x^4\rho(x) =  \int\limits_{-k}^k x^4\rho(x)dx + \int\limits_{|x|>k} x^4\rho(x)dx\\
	&\geq \left(\int\limits_{-k}^k x^2\rho(x)dx\right)^2 + k^2\int\limits_{|x|>k} x^2\rho(x)dx\\
	&= h(k)^2 + k^2(1-h(k)).
	\end{align*}
	Optimizing over $k$, with the constraint $h(k) \leq \frac{4k^2}{12}$, shows that for the choice of $k = \frac{3}{2}$,\linebreak $\EE\left[Y^4\right] \geq \frac{9}{8}.$
	Since
	$$\mathrm{Var}(Y^2) = \EE\left[Y^4\right] - \EE\left[Y^2\right]^2 = \EE\left[Y^4\right] - 1,$$
	the claim is proven.
	\fi
\end{proof}
Using the lemma, we now prove Theorem \ref{thm: p2}.
\begin{proof}[Proof of Theorem \ref{thm: p2}]
	First, as in Lemma \ref{lem: uncondtional covariance}, the product structure of $\mu$ implies that $\Sigma_2(\mu)$, the covariance matrix of $W_{n,d}^2(\mu)$, is diagonal. Write $X = \left(X_{(1)},...,X_{(n)}\right).$ There are two types of elements on the diagonal:
	\begin{itemize}
		\item The first corresponds to elements of the form $\mathrm{Var}(X_{(i)}X_{(j)})$. For those elements, by independence, $\mathrm{Var}(X_{(i)}X_{(j)}) = 1$. 
		\item The other type of elements are of the form $\mathrm{Var}\left(X_{(i)}^2\right)$. By Lemma \ref{lem: square variance}, $\mathrm{Var}(X_i^2) \geq \frac{1}{100}$.
	 	\end{itemize}
 		So, if $P: (\RR^n)^{\otimes 2} \to \mathrm{Sym}^2(\RR^n)$ is the projection operator, we have that 
 		$$\norm{\Sigma^{-\frac{1}{2}}P}_{op}^2 \leq 100.$$ 			
 	The estimates \eqref{eq: norm log concave} and \eqref{eq: transport log concave} are valid here as well.
 	Thus, Theorem \ref{thm: main} implies the result.
\end{proof}
\section{Non-homogeneous sums} \label{sec:nonhom}
In this section we consider a slight variation on the law of $W_{n,d}^p(\mu)$. Specifically, we let $\alpha := \{\alpha_i\}_{i=1}^d \subset \RR^+$, be a sequence of positive numbers and $X_i \sim \mu$ be i.i.d. random vectors in $\RR^n$. Define $W_{\alpha,d}^p(\mu)$ as the law of the non-homogeneous sum,
\begin{equation} \label{eq:nonhomsum}
\frac{1}{\norm{\alpha}_2}\sum\limits_{i=1}^d \alpha_i \left(X_i^{\stensor p} - \EE\left[X_i^{\stensor p }\right]\right),
\end{equation}
where for $q > 0$, $\norm{\alpha}_q^q := \sum\limits_{i=1}^d \alpha_i^q$.
The marginal law $\widetilde{W}_{\alpha,d}^p(\mu)$ is defined accordingly. The case $\alpha_i \equiv 1$ corresponds to $W_{n,d}^p(\mu)$. As it turns out, controlling the Stein discrepancy of $\widetilde{W}_{\alpha,d}^p(\mu)$ poses no new difficulties and Theorem \ref{thm: main} may be readily adapted to deal with these laws as well. The basic observation is that the calculation in \eqref{eq: normalized stein sums} also applies to this case.\\

Indeed, in the general case, if $Y_i$ are i.i.d. isotropic random vectors with Stein kernel given by $\tau_Y$, then, $S_\alpha: = \frac{1}{\norm{\alpha}_2}\sum\limits_{i=1}^d \alpha_i Y_i$ is isotropic as well, and it has a Stein kernel given by
$$\tau_{S_\alpha}(x) = \frac{1}{\norm{\alpha}^2_2}\sum\limits_{i=1}^d\alpha_i^2\EE\left[\tau_Y(Y_i)|S_\alpha = x\right].$$
By repeating the calculations which led to \eqref{eq: CLT stein}, we may see
\begin{align*} 
S^2(S_\alpha)&\leq \EE\left[\norm{\tau_{S_\alpha}(S_\alpha)-\mathrm{Id}}_{HS}^2\right] = \EE\left[\norm{\frac{1}{\norm{\alpha}
	_2^2}\sum\limits_{i=1}^d\alpha_i^2\EE\left[\tau_Y(Y_i)-\mathrm{Id}|S_\alpha\right]}_{HS}^2\right]\nonumber\\
&\leq \frac{1}{\norm{\alpha}_2^4}\sum\limits_{i=1}^d\alpha_i^4\EE\left[\norm{\tau_Y(Y_i)-\mathrm{Id}\right]}_{HS}^2 = \frac{\norm{\alpha}_4^4}{\norm{\alpha}_2^2}\EE\left[\norm{\tau_Y(Y)-\mathrm{Id}}_{HS}^2\right],
\end{align*}
which implies
$$S^2(S_\alpha) \leq \frac{\norm{\alpha}_4^4}{\norm{\alpha}_2^4}S^2(Y).$$ 
Combining this inequality with Lemma \ref{lem: main lemma} we obtain the following variant of Theorem \ref{thm: main}.
\begin{theorem}
	With the same notations as in Theorem \ref{thm: main},
	$$S^2\left(A_*W_{\alpha,d}^p(\mu)\right) \leq 2\frac{\norm{\alpha}_4^4}{\norm{\alpha}_2^4}\left(\norm{A}_{op}^2p^4\cdot n\sqrt{\EE\left[\norm{X}_2^{8(p-1)}\right]}\sqrt{\EE\left[\norm{D\vphi(G)}_{op}^8\right]} +n^p\right).$$
\end{theorem}
Thus, all of our results apply to non-homogeneous sums as well. We state here only the analogue for uniformly log-concave measures as reference.
\begin{theorem} \label{thm:nonhomsum}
	Let $\mu$ be an isotropic $L$-uniformly log-concave measure on $\RR^n$ which is also unconditional. Denote $\Sigma^{-\frac{1}{2}} = \sqrt{\widetilde{\Sigma}_p(\mu)^{-1}}$, where $\widetilde{\Sigma}_p(\mu)$ is the covariance matrix of $\widetilde{W}_{\alpha,d}(\mu)$. Then, there exists a constant $C_p$, depending only on $p$, such that
	$$S^2\left(\Sigma^{-\frac{1}{2}}_*\widetilde{W}_{\alpha,d}^p(\mu)\right) \leq  \frac{C_p}{L^{4}}n^{2p-1}\frac{\norm{\alpha}_4^4}{\norm{\alpha}_2^4}.$$
\end{theorem}
By specializing to $\mu = \gamma$ and $p=2$, the theorem gives the same bound as in \eqref{eq:nonhomog}, which was obtained in \cite{eldan2016information}. \\

As noted in the introduction, when $p = 2$, the symmetric matrix defined by \eqref{eq:nonhomsum} can be realized as normalized version of a Gram matrix $\mathbb{X}\mathbb{X}^T$, where $\mathbb{X}$ is an $n \times d$ matrix with independent columns.\\

By taking a different perspective on Theorem \ref{thm:nonhomsum}, we now show that, in some special cases, one may also allow dependencies between the columns of $\mathbb{X}$.
Let $\Sigma$ be a $d \times d$ positive definite matrix and $\{X_i\}_{i=1}^n$ i.i.d random vectors in $\RR^d$ with common law $\mathcal{N}\left(0,\Sigma\right)$. Suppose that for every $i = 1,...,d$, $\Sigma_{i,i} = 1$ and define $\mathbb{X}_\Sigma$ to be an $n \times d$ matrix whose $i^{th}$ row equals $X_i$. So, the rows of $\mathbb{X}_\Sigma$ are independent while its columns might admit dependencies. Now, set
$$W_n(\Sigma) := \frac{1}{\sqrt{d}}\left(\mathbb{X}_\Sigma\mathbb{X}_\Sigma^T -d\cdot\mathrm{Id}\right).$$
Our result will apply by a change of variables. If $U$ is a $d \times d$ orthogonal matrix which diagonalizes $\Sigma$ the following identity holds:
$$\mathbb{X}_\Sigma\mathbb{X}_\Sigma^T = \left(\mathbb{X}_\Sigma U\right)\left(\mathbb{X}_\Sigma U\right)^T,$$
 with the columns of $\mathbb{X}_\Sigma U$ being independent. Specifically, the rows of $\mathbb{X}_\Sigma$ are given by $U^TX_i$. Thus, if $\{\alpha_i\}_{i=1}^d$ are the eigenvalues of $\Sigma$, then for every $i,j$, $(\mathbb{X}_\Sigma U)_{i,j} \sim \mathcal{N}(0, \alpha_j)$. This implies that $W^2_{\alpha,d}(\gamma)$ is the law of the upper triangular part of $\frac{\sqrt{d}}{\norm{\alpha}_2^2}W_n(\Sigma)$. So,
$$S^2\left(\frac{\sqrt{d}}{\norm{\alpha}_2^2} W_n(\Sigma)\right) \leq Cn^3 \frac{\norm{\alpha}_4^4}{\norm{\alpha}_2^4} = Cn^3\frac{\mathrm{Tr}\left(\Sigma^4\right)}{\mathrm{Tr}\left(\Sigma^2\right)^2}.$$
As a particular case, we can assume that the rows of $\mathbb{X}_\Sigma$ form a stationary Gaussian process. Let $s:\NN \to \RR$ be a function with $s(0) = 1$ and define a symmetric $d \times d$ matrix by $\left(\Sigma_s\right)_{i,j} = s(|i-j|)$. If $\Sigma_s$ is positive definite,
then the proof of Theorem 1.2 in \cite{nourdin2012normal} shows:
$$\mathcal{W}^2_1(W_{n}(\Sigma_s), G_s) \leq Cn^3\frac{1}{d^2} \sum\limits_{i,j,k,\ell=1}^d s(|i-j|)s(|j-k|)s(|k-\ell|)s(|\ell-i|),$$
where $G_s$ is the law of a Wigner matrix, normalized to have the same covariance structure as $W_{n}(\Sigma_s)$.
Since $s(0) = 1$, it is clear that $\mathrm{Tr}(\Sigma_s^2)^2 \geq d^2$ and we also have
$$\mathrm{Tr}(\Sigma_s^4) = \sum\limits_{i,j,k\ell=1}^d s(|i-j|)s(|j-k|)s(|k-\ell|)s(|\ell-i|).$$
Thus, our result is directly comparable to the one in \cite{nourdin2012normal}.
\bibliography{bib}{}
\bibliographystyle{abbrv}
\end{document}